\documentclass[12pt,reqno]{amsart}
\usepackage{amssymb, amsmath}
\newtheorem{thm}{Theorem}

\numberwithin{defn}{section}
\numberwithin{thm}{section}
\numberwithin{Lemma}{section}
\numberwithin{Corollary}{section}
\numberwithin{Example}{section}
\numberwithin{subsection}{section}
\numberwithin{Remark}{section}
\numberwithin{equation}{section}
\numberwithin{ppn}{section}
\begin{document}
\title[ Extension of  Newton-Steffenssen  method  . . .]
{ Extension of Newton-Steffenssen  method by Gejji-Jafari decomposition  Technique  for solving nonlinear equations}
\author{J. P. Jaiswal }
\date{}
\maketitle


\textbf{Abstract.} 
In this paper we extend Newton-Steffenssen method for solving nonlinear equations, introduced by Sharma [J.R. Sharma, A composite third order Newton-Steffenssen method for solving nonlinear equations, Appl. Math. Comput. 169 (2005), 242-246] by using the Gejji-Jafari  decomposition technique. 
Several numerical examples are given to illustrate the efficiency and performance of this new method. 
\\



\section{Introduction}
Solving nonlinear equations is one of the most important problems in numerical analysis. To solve nonlinear equations,
iterative methods such as Newton's method are usually used. Throughout this paper we consider iterative methods to find a simple root $\alpha$,  of a nonlinear equation $f(x)=0$, where $f:I\subset R \rightarrow R$ for an open interval $I$. Many variants of the Newton's method have been suggested in the literature by different techniques.
One of them is Adomain decomposition method which is used in \cite{Chun} and other literatures.  To implement Adomain decomposition method, one has to calculate Adomain polynomial, which is another difficult task. Other techniques have also their limitations.  To overcome these difficulties, a new decomposition technique is introduced by Gejji-Jafari in \cite{Gejji}. In this paper we use this technique to extend Newton-Steffenssen method introduced by  Sharma \cite{Sharma}.


\section{Gejji-Jafari decomposition  method}
Consider the nonlinear equation
\begin{eqnarray}\label{eqn:21}
f(x)=0.
\end{eqnarray}
Throughout the paper we assume that $f(x)$ has a simple root at $\alpha$ and $\gamma$ is an initial guess close to $\alpha$. Let us transform the nonlinear equation $(\ref{eqn:21})$ into the following canonical form:
\begin{eqnarray}\label{eqn:22}
x=c+N(x),
\end{eqnarray}
where $N(x)$ nonlinear operator and $c$ is a constant. The main idea of this technique is to look for a solution having the series form
\begin{eqnarray}\label{eqn:23}
x=\sum_{i=0}^{\infty}x_i.
\end{eqnarray}
The nonlinear operator N can be decomposed as
\begin{eqnarray}\label{eqn:24}
N(x)=N(x_0)+\sum_{i=1}^{\infty}\left\{N\left(\sum_{j=0}^{i}x_i\right)-N\left(\sum_{j=0}^{i-1}x_i\right)\right\}.
\end{eqnarray}
From equations $(\ref{eqn:23})$ and $(\ref{eqn:24})$, equation $(\ref{eqn:22})$ is equivalent to
\begin{eqnarray}\label{eqn:25}
\sum_{i=0}^{\infty}x_i=c+N(x_0)+\sum_{i=1}^{\infty}\left\{N\left(\sum_{j=0}^{i}x_i\right)-N\left(\sum_{j=0}^{i-1}x_i\right)\right\}.
\end{eqnarray}
Thus we have the following recurrence relation:
\begin{eqnarray}\label{eqn:26}
x_0&=&c, \nonumber\\
x_1&=&N(x_0),\nonumber\\
x_2&=&N(x_0+x_1)-N(x_0),\nonumber\\
x_2&=&N(x_0+x_1+x_2)-N(x_0+x_1),\nonumber\\
&.& \nonumber\\
&.& \nonumber\\
&.& \nonumber\\
x_{m+1}&=&N(x_0+x_1+. . .+x_m)-N(x_0+x_1+. . .+x_{m-1}).
\end{eqnarray}
Then
\begin{eqnarray}\label{eqn:27}
x_0+x_1+. . .+x_{m+1}=N(x_0+x_1+. . .+x_m); m=1,2,...,
\end{eqnarray}
and \begin{eqnarray}\label{eqn:28}
x=c+\sum_{i=1}^{\infty}x_i
\end{eqnarray}
 In \cite{Gejji} it is proved that the series $\sum_{i=0}^{\infty}x_i$ converges absolutely and uniformly to a unique
solution of $(\ref{eqn:22})$.

\section{ Extension of Newton-Steffenssen  method}
Consider the following coupled system:
\begin{eqnarray}\label{eqn:31}
f(\gamma)+(x-\gamma)\left[f'(\gamma)+\frac{g(x)}{(x-\gamma)}\right]=0,
\end{eqnarray}
\begin{eqnarray}\label{eqn:32}
g(x)=f(x)-f(\gamma)-f'(\gamma)(x-\gamma).
\end{eqnarray}
The  equation $(\ref{eqn:31})$ of the above system can be rewritten as
\begin{eqnarray}\label{eqn:33}
x=\gamma-\frac{f(\gamma)(x-\gamma)}{f'(\gamma)(x-\gamma)+g(x)}.
\end{eqnarray} 
Comparing Equations $(\ref{eqn:22})$, Ist of $(\ref{eqn:26})$ and $(\ref{eqn:33})$, we have
\begin{eqnarray}\label{eqn:34}
x_0=c=\gamma.
\end{eqnarray}
and
\begin{eqnarray}\label{eqn:34a}
N(x)=-\frac{f(\gamma)(x-\gamma)}{f'(\gamma)(x-\gamma)+g(x)}.
\end{eqnarray}
Note that $x$ is approximated by
\begin{eqnarray}\label{eqn:35}
X_m=x_0+x_1+. . .+x_m,
\end{eqnarray}
where $\lim_{m \rightarrow \infty} X_m=x$.

For $m=0$, 
\begin{eqnarray}\label{eqn:36}
x\approx X_0=x_0=c=\gamma.
\end{eqnarray}

For $m=1$, 
\begin{eqnarray}\label{eqn:37}
x\approx X_1=x_0+x_1=\gamma+N(x_0).
\end{eqnarray}
where $N(x_0)$ is to be calculate. From $(\ref{eqn:34a})$ we have
\begin{eqnarray}\label{eqn:38}
N(x_0)=-\frac{f(\gamma)(x_0-\gamma)}{f'(\gamma)(x_0-\gamma)+g(x_0)}.
\end{eqnarray}
Thus $(\ref{eqn:37})$ becomes
\begin{eqnarray}\label{eqn:39}
x\approx X_1=x_0+x_1=\gamma-\frac{f(\gamma)(x_0-\gamma)}{f'(\gamma)(x_0-\gamma)+g(x_0)}=x_0-\frac{f(x_0)}{f'(x_0)}.
\end{eqnarray}
which yields the famous Newton's method with second order convergence
\begin{eqnarray}\label{eqn:39a}
x_{n+1}=x_n-\frac{f(x_n)}{f'(x_n)}.
\end{eqnarray}

For $m=2$, 
\begin{eqnarray}\label{eqn:310}
x\approx X_2=x_0+x_1+x_2=\gamma+N(x_0+x_1).
\end{eqnarray}
where $N(x_0+x_1)$ is to be calculate. From $(\ref{eqn:34a})$, $(\ref{eqn:32})$ and $(\ref{eqn:39})$ we have
\begin{eqnarray}\label{eqn:311}
N(x_0+x_1)=\frac{f(\gamma)^2}{f'(\gamma)\{f(x_0+x_1)-f(\gamma)\}}.
\end{eqnarray}
Thus we have 
\begin{eqnarray}\label{eqn:312}
x\approx X_2=x_0+x_1+x_2=\gamma-\frac{f(\gamma)^2}{f'(\gamma)\{f(\gamma)-f(x_0+x_1)\}},
\end{eqnarray}
which gives the following well known Newton-Steffenssen method \cite{Sharma} with third order convergence
\begin{eqnarray}\label{eqn:313}
x_{n+1}=x_n-\frac{f(x_n)^2}{f'(x_n)\{f(x_n)-f(y_n)\}},
\end{eqnarray}
where $y_n=x_n-\frac{f(x_n)}{f'(x_n)}$.

For $m=3$, 
\begin{eqnarray}\label{eqn:313}
x\approx X_3=x_0+x_1+x_2+x_3=\gamma+N(x_0+x_1+x_2).
\end{eqnarray}
where $N(x_0+x_1+x_2)$ is to be calculate. From $(\ref{eqn:34a})$, $(\ref{eqn:32})$ and $(\ref{eqn:312})$ we have
\begin{eqnarray}\label{eqn:314}
N(x_0+x_1+x_2)=\frac{f(\gamma)^3}{f'(\gamma)\{f(\gamma)-f(x_0+x_1)\}\{f(x_0+x_1+x_2)-f(\gamma)\}},
\end{eqnarray}
Thus we have
\begin{eqnarray}\label{eqn:315}
x\approx X_3&=&x_0+x_1+x_2+x_3 \nonumber\\
&=&\gamma-\frac{f(\gamma)^3}{f'(\gamma)\{f(\gamma)-f(x_0+x_1)\}\{f(\gamma)-f(x_0+x_1+x_2)\}},.
\end{eqnarray}
which suggests the following three-step iterative method
\begin{eqnarray}\label{eqn:316}
y_n&=&x_n-\frac{f(x_n)}{f'(x_n)},\nonumber\\
z_n&=&x_n-\frac{f(x_n)^2}{f'(x_n)\{f(x_n)-f(y_n)\}},\nonumber\\
x_{n+1}&=&x_n-\frac{f(x_n)^3}{f'(x_n)\{f(x_n)-f(y_n)\}\{f(x_n)-f(z_n)\}}.
\end{eqnarray}
Similarly we can obtain higher-order iterative methods. For general $n$ it can be shown that the $(n-1)$-step iterative method is
\begin{eqnarray}\label{eqn:316a}
a1_n&=&x_n-\frac{f(x_n)}{f'(x_n)},\nonumber\\
a2_n&=&x_n-\frac{f(x_n)^2}{f'(x_n)\{f(x_n)-f(a1_n)\}},\nonumber\\
a3_n&=&x_n-\frac{f(x_n)^3}{f'(x_n)\{f(x_n)-f(a1_n)\}\{f(x_n)-f(a2_n)\}}\nonumber\\
&.& \nonumber\\
&.& \nonumber\\
&.& \nonumber\\
a(n-2)_n&=&x_n-\frac{f(x_n)^{n-2}}{f'(x_n)\{f(x_n)-f(a1_n)\}\{f(x_n)-f(a2_n)\}. . . \{f(x_n)-f(a(n-3)_n)\}}\nonumber\\
x_{n+1}&=&x_n-\frac{f(x_n)^{n-1}}{f'(x_n)\{f(x_n)-f(a1_n)\}\{f(x_n)-f(a2_n)\}. . . \{f(x_n)-f(a(n-2)_n)\}}.
\end{eqnarray}
\\
Now we prove that order of convergence of the iterative method $(\ref{eqn:316})$ is four, which is shown by the following theorem:

\begin{thm}
Let $\alpha \in I$ be a simple zero of a sufficiently differentiable function $f:I\subseteq \Re \rightarrow \Re$ in an open interval I. If $x_0$ is sufficiently close to $\alpha$, then the three-step method  defined by $(\ref{eqn:316})$ has order  fourth-order convergence.
\end{thm}
\begin{proof}
By applying the Taylor series expansion theorem and taking account $f(\alpha)=0$, we can write 
\begin{equation}\label{eqn:317}
f(x_n)=e_n+c_2e_n^2+c_3e_n^3+c_4e_n^4+c_5e_n^5+c_6e_n^6+c_7e_n^7+c_8e_n^8+O(e_n^9),
\end{equation}
where $c_k=\frac{f^{k}(\alpha)}{\lfloor k},\ k=1, 2, . . .$ and $e_n$ be the error in $x_n$ after $n$ iterations  i.e. $e_n=x_n-\alpha$;
\begin{eqnarray}\label{eqn:318}
f'(x_n) &=&f'(\alpha)[1+2c_2e_n+3c_3e_n^2+4c_4e_n^3+5c_5^5e_n^4\nonumber\\
            &&+6c_6e_n^5+7c_7e_n^6+8c_8e_n^7+O(e_n^{8})].
\end{eqnarray}
By considering the above relations, one can obtain
\begin{equation}\label{eqn:319}
y_n=\alpha+c_2e_n^2+2(c_3-c_2^2)e_n^3+......+O(e_n^{8}).
\end{equation}
At this time, we should expand $f(y_n)$ around the exact $\alpha$ root by taking into consideration $(\ref{eqn:319})$. Accordingly, we have
\begin{equation}\label{eqn:320}
f(y_n)=f'(\alpha)[c_2e_n^2+2(-c_2^2+c_3)e_n^3+......+O(e^{12})].
\end{equation}
From $(\ref{eqn:317})$, $(\ref{eqn:318})$ and $(\ref{eqn:320})$ it can found that
\begin{equation}\label{eqn:321}
z_n=\alpha+c_2^2e^3+(3c_2c3-3c_2^3)e^4+ . . .+O(e_n^{8}).
\end{equation}
Now expand $f(z_n)$ around the exact $\alpha$ root, we have
\begin{equation}\label{eqn:322}
f(z_n)=f'(\alpha)[c_2^2e^3+(3c_2c3-3c_2^3)e^4+ . . .+O(e_n^{8})].
\end{equation}
Finally by virtue of $(\ref{eqn:317})$, $(\ref{eqn:318})$, $(\ref{eqn:320})$ and $(\ref{eqn:322})$ it can be obtained that
\begin{equation}\label{eqn:322}
e_{n+1}=c_2^3e^4+O(e_n^{5}).
\end{equation}

\end{proof}
Similarly we can prove that iterative method $(\ref{eqn:316a})$ has $n^{th}$-order convergence.  
  

\section{ Numerical Testing}
Here we consider, the following eight test functions to illustrate the accuracy of new iterative method. Some of them are taken from \cite{Neta} and some from \cite{Gupta}. The root of each nonlinear test function is also listed. All the computations reported here we have done using Mathematica 8. Scientific computations in many branches of science and technology demand very high precision degree of numerical precision.  We consider the number of decimal places as follows: 10000 digits floating point (SetAccuraccy=10000) with  SetAccuraccy Command. In examples considered in this article, the stopping criterion is the $\left|f(x_n)\right|\leq \epsilon$  where  $\epsilon=10^{-10000}$. The test non-linear functions are listed in Table-1.

Here we comparer performance of our new method $(\ref{eqn:316})$ to the methods of Yun (YN) \cite{Yun}, of Chun (CN) \cite{Chun} and Noor(NR)  \cite{Noor}. The results of comparison for the test function are provided in the Table 2. It can be seen that the resulting method from our class are accurate and efficient in terms of number of accurate decimal places to find the roots after some iterations. 


\small \begin{table}[htb]
 \caption{ Test functions and their roots.}
  \begin{tabular}{ll} \hline
Non-linear function                             & \hspace{50pt} Roots                  \\ \hline 
$f_1(x)=\sin^2x-x^2+1$                          & \hspace{50pt} 1.4044916482153412260  \\ 
$f_2(x)=x^2-e^x-3x+2$                           & \hspace{50pt} 0.25753028543986076046 \\ 
$f_3(x)=(x-1)^3-1$                              & \hspace{50pt} 2                      \\ 
$f_4(x)=x^3-10$                                 & \hspace{50pt} 2.1544346900318837218  \\ 
$f_5(x)=xe^x-\sin^2x+3cosx+5$                   & \hspace{50pt} -1.2076478271309189270 \\ 
$f_6(x)=e^{x^2+7x-30}-1$                        & \hspace{50pt} 3                      \\ 
$f_7(x)=x^2+\sin x+x$                           & \hspace{50pt} 0                      \\ 
$f_8(x)=\sin(2\cos x)-1-x^2+e^{\sin x^3}$       & \hspace{50pt} 1.3061752018468278250  \\ 
 \hline
  \end{tabular}
  \label{tab:abbr}
\end{table}

\newpage
\footnotesize
\begin{table}[htb]
 \caption{Comparison of different methods with the same total number of function evaluations (TNFE = 24)}
  \begin{tabular}{lllllllll} \hline
\tiny {$ Test Function$}    & $Guess$          &$CH$         &$YN$         &$NR$        &$(\ref{eqn:316})$ \\ \hline 
$\left|f_1\right|$ & -1.0      &0.15403e-267  &0.41715e-642  &0.52856e-2   &0.22708e-1576  \\ 
$$                 &  2.0      &0.35940e-1457 &0.23037e-1530 &0.11251e-3   &0.13526e-2046 \\
$$                 &  1.0      &0.15403e-267  &0.41715e-642  &0.52856e-2   &0.22708e-1576  \\   \hline

$\left|f_2\right|$ & 2.0       &0.74095e-1843 &0.42734e-1881 &0.16354e-4   &0.62927e-2949  \\ 
$$                 & 2.5       &0.70566e-1411 &0.31383e-1450 &0.199813-3   &0.18307e-1924  \\
$$                 & -1.5      &0.91945e-2055 &0.25223e-2127 &0.46951e-4   &0.13343e-2794 \\  \hline

$\left|f_3\right|$ & 3.5       &0.95901e-373  &0.60165e-400  &0.43860e-1   &0.20884e-671 \\ 
$$                 & 3.1       &0.22428e-546  &0.32703e-582  &0.13254e-1   &0.36490e-921 \\ 
$$                 & 1.5       &0.39254e+1    &0.21668e+1    &0.23818e+1   &0.38950e-663 \\ \hline

$\left|f_4\right|$ & 1.5       &0.610763-350  &0.25884e-740  &0.27421e-1   &0.47059e-1750  \\ 
$$                 & 1.2       &0.38918e-2    &0.12361e-298  &0.85122e+0   &0.716731e-925 \\
$$                 & 1.0       &0.749233+3    &0.27646e+5    &0.42601e+1   &0.18324e-518 \\  \hline

$\left|f_5\right|$ & -2.0      &0.74075e-140  &0.16833e-150  &0.10709e+1   &0.60969e-322 \\ 
$$                 & -1.5      &0.19698e-1039 &0.21190e-1089 &0.26108e-2   &0.52642e-1607\\ 
$$                 & -1.0      &0.19824e-693  &0.10164e-1074 &0.38167e-2   &0.49335e-2031 \\ \hline

$\left|f_6\right|$ & 3.5       &0.45434e-5    &0.99135e-6    &0.13063e+2   &0.15256e-22 \\  
$$                 & 3.2       &0.11804e-200  &0.14326e-215  &0.92374e-1   &0.69853e-422 \\ 
$$                 & 2.9       &0.20627e+114  &Indeterminate &0.42942e+0   &0.42509e-496 \\ \hline

$\left|f_7\right|$ & 0.3       &0.51508e-2921   &0.46046e-3026   &0.14553e-6   &0.13490e-3702 \\  
$$                 & 0.1       &0.10912e-4558   &0.46696e-4679   &0.77208e-10  &0.16080e-5443 \\ 
$$                 & -0.2      &0.27239e-2693   &0.16105e-2868   &0.19173e-6   &0.17394e-3843 \\ 
\hline

$\left|f_8\right|$ & 1.35      &0.29081e-4253   &0.15643e-4384   &0.62249e-9    &0.15434e-5050 \\  
$$                 & 1.31      &0.37148e-8008   &0.11024e-8139   &0.20749e-16   &0.66987e-8948 \\ 
$$                 & 1.29      &0.56923e-5177   &0.37152e-5314   &0.56633e-11   &0.20166e-6195\\ 
\hline

%

  \end{tabular}
  \label{tab:abbr}
\end{table}


\textsc{Jai Prakash Jaiswal\\
Department of Mathematics, \\
Maulana Azad National Institute of Technology,\\
Bhopal, M.P., India-462051}.\\
E-mail: { asstprofjpmanit@gmail.com; jaiprakashjaiswal@manit.ac.in}.\\\\

\begin{thebibliography}{10}
\bibitem{Gejji}
V. Daftardar-Gejji and H. Jafari: An iterative method for solving nonlinear functional equations, Journal of Mathematical Analysis and Applications, 316 (2006), 753-763.
\bibitem{Sharma}
J. R. Sharma: A composite third order Newton–Steffenssen method for solving nonlinear equations, Appl. Math. Comput., 169 (2005), 242-246.
\bibitem{Yun}
J. H. Yun: A note on three-step iterative method for nonlinear equations, Applied Mathematics and Computation, 202 (2008), 401-405.
\bibitem{Chun}
C. Chun: Iterative methods improving Newton's method by the decomposition method, Computers and Mathematics with Applications, 50 (2005), 1559-1568.
\bibitem{Noor}
M. A. Noor, K. I. Noor,  E. Al-Said and M. Waseem: Some New Iterative methods for nonlinear equations, Mathematical Problems in Engineering,  2010, Article ID 198943, 12 pages.
\bibitem{Chun}
C. Chun: Iterative methods improving Newton's method by the decomposition method,  Comput. Math. Appl. 50 (2005), 1559-1568
\bibitem{Neta}
C. Chun and B. Neta: A new sixth-order scheme for nonlinear equations, Applied Mathematics Letters 25 (2012), 185-189.
\bibitem{Gupta}
J. R. Sharma, R. K. Guha and P. Gupta: Improved King's methods with optimal order of convergence based on rational approximations, Applied Mathematics Letters 26 (2013), 473-480. 
\\
\end{thebibliography}
\end{document}